\documentclass[12pt,english]{article}
\usepackage[T1]{fontenc}
\usepackage[latin9]{inputenc}
\usepackage{geometry}
\usepackage{babel}
\usepackage{amsthm}
\usepackage{amsmath}
\usepackage{amssymb}
\usepackage[unicode=true]{hyperref}
\usepackage{breakurl}
\usepackage{mathtools}

\usepackage{color} %for comments

\theoremstyle{plain}
\newtheorem{thm}{Theorem}

\theoremstyle{plain}
\newtheorem{lem}[thm]{Lemma}
  
\theoremstyle{plain}
\newtheorem{cor}[thm]{Corollary}

\usepackage{titlesec}
\titleformat{\section}
  {\normalfont\large\bfseries}{\thesection.}{.8ex}{} %changed from Large
\titleformat{\subsection}
  {\normalfont\bfseries}{\thesubsection.}{.8ex}{} %large removed
\renewcommand{\S}{\mathfrak S}
\newcommand{\binomE}[2]{\binom{#1}{#2}_E}
\newcommand{\up}{\kern -2.7pt\uparrow}
\newcommand{\h}{\mathbf h}
\renewcommand{\r}{\mathbf r}
\newcommand{\Sym}{\mathbf{Sym}}
\renewcommand{\Bar}{\,|\,}
\DeclareMathOperator{\des}{des}
\DeclareMathOperator{\altdes}{altdes}
\DeclareMathOperator{\maj}{maj}
\DeclareMathOperator{\altmaj}{altmaj}
\newcommand{\leftprod}[1]{\overset{\longleftarrow}{\prod_{#1}}}

\begin{document}

\title{Counting Permutations by Alternating Descents}

\author{Ira M. Gessel\\
Department of Mathematics\\
Brandeis University\\
\texttt{\href{mailto:gessel@brandeis.edu}{gessel@brandeis.edu}}\and
Yan Zhuang\\
Department of Mathematics\\
Brandeis University\\
\texttt{\href{mailto:zhuangy@brandeis.edu}{zhuangy@brandeis.edu}}}
\maketitle
\begin{abstract}

We find the exponential generating function 
for permutations with all valleys even and all peaks odd, and use
it to determine the asymptotics for its coefficients,  answering a
question posed by Liviu Nicolaescu. The generating function can be expressed
as the reciprocal of a  sum involving Euler numbers: 
\[
\left(1-E_1x+E_{3}\frac{x^{3}}{3!}-E_{4}\frac{x^{4}}{4!}+E_{6}\frac{x^{6}}{6!}-E_{7}\frac{x^{7}}{7!}+\cdots\right)^{-1},
\tag{$*$}
\]
where $\sum_{n=0}^\infty E_n x^n\!/n! = \sec x + \tan x$.
We give two proofs of this formula. The first uses a system of differential equations whose solution gives the generating function
\begin{equation*}
\frac{3\sin\left(\frac{1}{2}x\right)+3\cosh\left(\frac{1}{2}\sqrt{3}x\right)}{3\cos\left(\frac{1}{2}x\right)-\sqrt{3}\sinh\left(\frac{1}{2}\sqrt{3}x\right)},
\end{equation*} 
which we then show is equal to $(*)$. The second proof derives $(*)$ directly from general permutation enumeration techniques, using noncommutative symmetric functions. The generating function $(*)$ is an ``alternating'' analogue of David and Barton's generating function 
\[
\left(1-x+\frac{x^{3}}{3!}-\frac{x^{4}}{4!}+\frac{x^{6}}{6!}-\frac{x^{7}}{7!}+\cdots\right)^{-1},
\]
for permutations with no increasing runs of length 3 or more. Our general results give  further  alternating analogues of permutation enumeration formulas, including results of  Chebikin and Remmel.
\end{abstract}

\section{Introduction}

This paper was inspired by a question posed by Liviu
Nicolaescu on discrete Morse functions arising in combinatorial
topology \cite{Nicolaescu2012}, which reduces to the following combinatorial
problem. Let $\pi=\pi_{1}\pi_{2}\cdots\pi_{n}$ be a permutation in
$\mathfrak{S}_{n}$, the set of permutations of  $[n]=\{1,2,\dots,n\}$ (or more generally, any sequence of distinct integers). We say
that $i$ is a \emph{peak} of $\pi$ if $\pi_{i-1}<\pi_{i}>\pi_{i+1}$
and that $j$ is a \emph{valley} of $\pi$ if $\pi_{j-1}>\pi_{j}<\pi_{j+1}$.
For example, if $\pi=4762315$, then its peaks are 2 and 5, and its
valleys are 4 and 6. 
Let $f(n)$ be the
number of permutations in  $\mathfrak{S}_{n}$ that have all valleys even and all peaks odd. What is the
behavior of $f(n)/n!$ as $n\rightarrow\infty$?

To answer this question, we find the exponential generating function
for permutations with all valleys even and all peaks odd and use it to
derive an asymptotic formula for $f(n)/n!$,  thus answering
Nicolaescu's question. 

In Section \ref{s-2}, we show that the exponential generating function $F(x)$ for $f(n)$ is given by the formula
\begin{equation}
\label{e-F1}
F(x) =\frac{3\sin\left(\frac{1}{2}x\right)+3\cosh\left(\frac{1}{2}\sqrt{3}x\right)}{3\cos\left(\frac{1}{2}x\right)-\sqrt{3}\sinh\left(\frac{1}{2}\sqrt{3}x\right)}.
\end{equation}
by a finding recurrences for $f$ and for an auxiliary sequence, and solving the corresponding system of differential equations. It follows from this generating function that $f(n)/n!$ is asymptotic to $2\beta^{n+1}$, where $\beta = .7693323708\cdots$ is the zero of the denominator of $F(x)$ of smallest modulus.

In Section \ref{s-3}, we
show that 
\begin{equation}
\label{e-F2}
F(x) = 
\left(1-E_1x+E_{3}\frac{x^{3}}{3!}-E_{4}\frac{x^{4}}{4!}+E_{6}\frac{x^{6}}{6!}-E_{7}\frac{x^{7}}{7!}+\cdots\right)^{-1},
\end{equation}
where the Euler numbers $E_n$ are defined by 
$\sum_{n=0}^\infty E_n x^n\!/n! = \sec x + \tan x$.
Although we can prove directly that \eqref{e-F1} and \eqref{e-F2} are equivalent, we also give a more conceptual direct proof of \eqref{e-F2}.
First we recall a result of David and Barton \cite[pp.~156--157]{David1962}:
the exponential generating
function for permutations that avoid the increasing consecutive pattern
$12\cdots m$ (i.e., permutations in which every increasing run has length less than $m$) is  
\begin{align*}
\left(1-x+\frac{x^{m}}{m!}-\frac{x^{m+1}}{\left(m+1\right)!}+\frac{x^{2m}}{\left(2m\right)!}-\frac{x^{2m+1}}{\left(2m+1\right)!}+\cdots\right)^{-1},
\end{align*}
which for $m=3$ is 
\begin{equation}
\label{e-m=3}
\left(1-x+\frac{x^{3}}{3!}-\frac{x^{4}}{4!}+\frac{x^{6}}{6!}-\frac{x^{7}}{7!}+\cdots\right)^{-1}.
\end{equation}

We can explain the similarity between \eqref{e-F2} and \eqref{e-m=3} by using the concept of ``alternating descents'' introduced by Denis Chebikin \cite{chebikin}. We say that $i\in [n-1]$ is a \emph{descent} of a  permutation $\pi$ in $\S_n$ if $\pi_i>\pi_{i+1}$.  An \emph{increasing run} of $\pi$ is a maximal consecutive subsequence of $\pi$ containing no descents. So the increasing runs of $253146$ are $25$, $3$, and $146$. Note that the number of increasing runs of a nonempty permutation is one more than the number of descents.

Following Stanley \cite{st-alt}, we say that $\pi=\pi_{1}\pi_{2}\cdots\pi_{n}$
is an \emph{alternating }permutation if $\pi_{1}>\pi_{2}<\pi_{3}>\pi_{4}<\cdots$.
If instead $\pi_{1}<\pi_{2}>\pi_{3}<\pi_{4}>\cdots$, then we say
that $\pi$ is \emph{reverse-alternating}. 
It is well known that the number of alternating permutations of $[n]$ is $E_n$.
Alternating permutations of length $n$ are in bijection with reverse-alternating
permutations of length $n$ via complementation; that is, if $\pi=\pi_{1}\pi_{2}\cdots\pi_{n}$
is alternating, then its complement 
\[
\pi^{c}=\left(n+1-\pi_{1}\right)\left(n+1-\pi_{2}\right)\cdots\left(n+1-\pi_{n}\right)
\]
is reverse-alternating, and vice versa. Thus, the Euler numbers count
reverse-alternating permutations as well.

Following Chebikin, we say that $i\in [n-1]$ is an \emph{alternating descent} of $\pi$ if $i$ is odd and $\pi_i>\pi_{i+1}$ or if $i$ is even and $\pi_i<\pi_{i+1}$. We define an \emph{alternating run}%
  \footnote{The term ``alternating run'' has been used for a different concept (see Definition 1.37 
   in \cite{Bona2012}, or   \cite{oeisA059427}), but our usage here should cause no confusion.} 
of $\pi$ to be a maximal consecutive subsequence of $\pi$ containing no alternating descents. Thus the alternating runs of the permutation $3421675$ are $342$, $1$, and $675$. An alternating run that starts in an odd position will be a reverse-alternating permutation and an alternating run that starts in an even position will be an alternating permutation, and the number of alternating runs of a permutation will be one more than the number of alternating descents. 

A peak in an even position of a permutation $\pi$ will correspond to a subsequence $\pi_{2i-1}<\pi_{2i}>\pi_{2i+1}$ and thus will be contained in an alternating run of length  at least~3, and similarly a valley in an odd position must also be contained in an alternating run of length at least 3. Conversely, any alternating run of length at least 3 contains (as its second entry) either a peak in an even position or a valley in an odd position. Thus the permutations with all valleys even and all peaks odd are the same as the permutations in which every alternating run has length less than 3.

Then the similarity between \eqref{e-F2} and \eqref{e-m=3} can be explained by a general principle that whenever we have a formula like \eqref{e-m=3} that counts permutations with a restriction on increasing run lengths, there is an analogous formula, with $x^n\!/n!$ replaced by $E_n x^n\!/n!$, for counting permutations with the same restriction on alternating run lengths. We will make this idea precise in Section 3, using basic facts about noncommutative symmetric functions \cite{ncsf1}. In particular, we obtain Chebikin's formula \cite{chebikin} on ``alternating Eulerian polynomials'' that count permutations by the number of alternating descents, and Remmel's formula \cite{remmel} on counting permutations by the number of alternating descents and ``alternating major index.''

\section{Permutations with All Valleys Even and All Peaks Odd}
\label{s-2}

\subsection{The exponential generating function}

In this section, we use recurrrences to derive the exponential generating function for
permutations with all valleys
even and all peaks odd.
\begin{thm}
\label{t-m=00003D3}Let $F$ be the exponential generating function
for permutations with all valleys even and all peaks odd. Then $F$
is given by 
\begin{equation}
\label{e-Fgf}
F(x)  =\frac{3\sin\left(\frac{1}{2}x\right)+3\cosh\left(\frac{1}{2}\sqrt{3}x\right)}{3\cos\left(\frac{1}{2}x\right)-\sqrt{3}\sinh\left(\frac{1}{2}\sqrt{3}x\right)}.
\end{equation}
The first few values of $f(n)$ are as follows:
\[\vbox{\halign{\ \hfil\strut$#$\hfil\ \vrule&&\hfil\ $#$\ \hfil\cr
n&0&1&2&3&4&5&6&7&8&9&10&11&12\cr
\noalign{\hrule}
f(n)&1&1&2&4&13&50&229&1238&7614&52706&405581&3432022&31684445\cr}}
\]

\end{thm}
To prove Theorem \ref{t-m=00003D3}, we first obtain recurrence relations for the coefficients of $F(x)$.
\begin{lem}
Let $f(n)$ be the number of permutations in $\S_n$ with all valleys even
and all peaks odd, and let $g(n)$ be the number of such permutations
that end with an ascent. Then 
\begin{equation}
f(n+1)=\sum_{k=0}^{\left\lfloor \frac{n-1}{2}\right\rfloor }\binom{n}{2k}f(2k)f(n-2k)+\begin{cases}
f(n), & \text{if $n$ is even,}\\
g(n), & \text{if $n$ is odd,}
\end{cases}\label{e-recf}
\end{equation}
and 
\begin{equation}
g(n+1)=\sum_{k=0}^{\left\lfloor \frac{n-2}{2}\right\rfloor }\binom{n}{2k}f(2k)\left[f(n-2k)-g(n-2k)\right]
+\begin{cases}
  f(n), & \text{if $n$ is even,}\\
  g(n), & \text{if $n$ is odd,}
\end{cases}\label{e-recg}
\end{equation}
with $f(0)=g(0)=1$.\end{lem}
\begin{proof}
The initial conditions are true by convention. Let $\pi=\pi_{1}\pi_{2}\cdots\pi_{n+1}$ be an $\left(n+1\right)$-permutation
with all peaks odd and all valleys
even. We remove the largest letter $n+1$ to obtain two words with
distinct letters, the subword $\pi^{\prime}$ of $\pi$ of all letters
before $n+1$ and the subword $\pi^{\prime\prime}$ of $\pi$ of all
letters after $n+1$. Apply the standard reduction map to $\pi^{\prime}$
and $\pi^{\prime\prime}$, replacing the $i$th smallest letter in
each word by $i$, so that both words become permutations with the
order relations among the letters in each word preserved. By a slight
abuse of notation, call these permutations $\pi^{\prime}$ and $\pi^{\prime\prime}$. Let us consider  three cases:
\begin{itemize}
\item If $\pi_{1}=n+1$, then $\pi^{\prime}$ is empty and $\pi^{\prime\prime}$
is an $n$-permutation with all peaks even and all valleys odd, but
such $n$-permutations are in bijection (by complementation) with
$n$-permutations with all peaks odd and all valleys even. Hence,
there are $f(n)$ possibilities for $\pi^{\prime\prime}$ and thus for
$\pi$, which corresponds to the case $k=0$ in (\ref{e-recf}). 
\item Now, suppose that $\pi_{i}=n+1$, where $1<i<n+1$. Note that $i$
is a peak and so it must be odd, and it follows that $\pi^{\prime}$
is of length $2k$ for some $1\leq k\leq\left\lfloor \frac{n-1}{2}\right\rfloor $.
There are $\binom{n}{2k}$ ways to choose the letters of $\pi^{\prime}$,
and $f(2k)$ choices for $\pi^{\prime}$ given a set of $2k$ letters
because $\pi^{\prime}$ must also have all peaks odd and all valleys
even. Then $\pi^{\prime\prime}$ is composed of the remaining $n-2k$
letters, and can be formed in $f(n-2k)$ ways because $\pi^{\prime\prime}$
has all peaks even and all valleys odd. Hence, this corresponds to
the case $1\leq k\leq\left\lfloor \frac{n-1}{2}\right\rfloor $ in
(\ref{e-recf}). 
\item Finally, consider the case in which $\pi_{n+1}=n+1$. Then $\pi^{\prime\prime}$
is empty, and if $n$ is even, then $\pi^{\prime}$ can be any $n$-permutation
with all peaks odd and all valleys even. However, if $n$ is odd,
then $\pi^{\prime}$ must also have the property that it ends with
an ascent; otherwise, if it ends with a descent, then $\pi$ has a
valley in an odd position. Therefore, this gives us the final summand
in (\ref{e-recf}).
\end{itemize}
The reasoning for (\ref{e-recg}) is similar.
There are only two differences to explain. First, the $f(n-2k)$ term
is replaced with $f(n-2k)-g(n-2k)$ because $\pi^{\prime\prime}$
is an $(n-2k)$-permutation with all peaks even and all valleys odd and
ending with an ascent, which are in bijection with $(n-2k)$-permutations
with all peaks odd and all valleys even and ending with a descent.
Second, we sum up to $\left\lfloor \frac{n-2}{2}\right\rfloor $ rather
than $\left\lfloor \frac{n-1}{2}\right\rfloor $ because $n+1$ cannot
be the penultimate letter; otherwise, the permutation would end with
a descent rather than an ascent.
\end{proof}
Next we convert these two recurrences into differential
equations involving the exponential generating functions for $\left\{ f(n)\right\} _{n\geq0}$
and $\left\{ g(n)\right\} _{n\geq0}$. It is helpful to first
split these two sequences into odd parts and even parts.
\begin{lem}
Let $\left\{ f_{1}(n)\right\} _{n\geq0}$ be the sequence of odd terms
of $\left\{ f(n)\right\} _{n\geq0}$, let $\left\{ f_{2}(n)\right\} _{n\geq0}$
be the sequence of even terms of $\left\{ f(n)\right\} _{n\geq0}$, let
$\left\{ g_{1}(n)\right\} _{n\geq0}$ be the sequence of odd terms
of $\left\{ g(n)\right\} _{n\geq0}$, and let $\left\{ g_{2}(n)\right\} _{n\geq0}$
be the sequence of even terms of $\left\{ g(n)\right\} _{n\geq0}$.
That is, 
\[
f_{1}(n)=\begin{cases}
f(n), & \text{if $n$ is odd,}\\
0, & \text{if $n$ is even,}
\end{cases}
\]
\[
f_{2}(n)=\begin{cases}
0, &\text{if $n$ is odd,}\\
f(n), & \text{if $n$ is even,}
\end{cases}
\]
and similarly for $g_{1}\left(n\right)$ and $g_{2}\left(n\right)$.
Then the system of two recurrences in the previous lemma is equivalent
to the system of four recurrences 
\[
f_{1}(n+1)=\sum_{k=0}^{n-1}\binom{n}{k}f_{2}(k)f_{2}(n-k)+f_{2}(n),
\]
\[
f_{2}(n+1)=\sum_{k=0}^{n-1}\binom{n}{k}f_{2}(k)f_{1}(n-k)+g_{1}(n),
\]

\[
g_{1}(n+1)=\sum_{k=0}^{n-1}\binom{n}{k}f_{2}(k)\left[f_{2}(n-k)-g_{2}(n-k)\right]+f_{2}(n),
\]
and 
\[
g_{2}(n+1)=\sum_{k=0}^{n-2}\binom{n}{k}f_{2}(k)\left[f_{1}(n-k)-g_{1}(n-k)\right]+g_{1}(n)
\]
with $f_{1}(0)=g_{1}(0)=0$ and $f_{2}(0)=g_{2}(0)=1$.\end{lem}
\begin{proof}
We have 
\[
f_{1}(2n+1)=\sum_{k=0}^{n-1}\binom{2n}{2k}f_{2}(2k)f_{2}(2n-2k)+f_{2}(2n),
\]
and since $f_{2}(m)=0$ for $m$ odd, this is equivalent to 
\[
f_{1}(n+1)=\sum_{k=0}^{n-1}\binom{n}{k}f_{2}(k)f_{2}(n-k)+f_{2}(n).
\]
The other three recurrences are derived in the same way.
\end{proof}
Now, we convert this system of four recurrences into a system of four
differential equations using standard generating function methods.
\begin{lem}
Let $F_{1}$, $F_{2}$, $G_{1}$, and $G_{2}$ be the exponential
generating functions for the sequences $\left\{ f_{1}(n)\right\} _{n\geq0}$,
$\left\{ f_{2}(n)\right\} _{n\geq0}$, $\left\{ g_{1}(n)\right\} _{n\geq0}$,
and $\left\{ g_{2}(n)\right\} _{n\geq0}$. Then these
generating functions satisfy the system of differential equations
\begin{align*}
F_{1}^{\prime}(x)&=F_{2}^{2}(x),\\
F_{2}^{\prime}(x)&=F_{1}(x)F_{2}(x)+G_{1}(x),\\
G_{1}^{\prime}(x)&=F_{2}(x)\left[F_{2}(x)-G_{2}(x)\right]+F_{2}(x),\\
\shortintertext{and}
G_{2}^{\prime}(x)&=F_{2}(x)\left[F_{1}(x)-G_{1}(x)\right]+G_{1}(x)
\end{align*}
with initial conditions $F_{1}(0)=G_{1}(0)=0$ and\ $F_{2}(0)=G_{2}(0)=1$. 

Moreover, the solution of this system is given by
\[
F_{1}(x)=\frac{\sqrt{3}\sinh\left(\sqrt{3}x\right)+3\sin x}{3\cos x+4-\cosh\left(\sqrt{3}x\right)},
\]

\[
F_{2}(x)=\frac{2\sqrt{3}\sin\left(\frac{1}{2}x\right)\sinh\left(\frac{1}{2}\sqrt{3}x\right)+6\cos\left(\frac{1}{2}x\right)\cosh\left(\frac{1}{2}\sqrt{3}x\right)}{3\cos x+4-\cosh\left(\sqrt{3}x\right)},
\]

\[
G_{1}(x)=\frac{4\sqrt{3}\cos\left(\frac{1}{2}x\right)\sinh\left(\frac{1}{2}\sqrt{3}x\right)}{3\cos x+4-\cosh\left(\sqrt{3}x\right)},
\]
and 
\[
G_{2}(x)=\frac{6\cos\left(\frac{1}{2}x\right)\cosh\left(\frac{1}{2}\sqrt{3}x\right)+2\sqrt{3}\sin\left(\frac{1}{2}x\right)\sinh\left(\frac{1}{2}\sqrt{3}x\right)+2-2\cosh\left(\sqrt{3}x\right)}{3\cos x+4-\cosh\left(\sqrt{3}x\right)}.
\]
\end{lem}
\begin{proof}
Let us take the first recurrence, 
\[
f_{1}(n+1)=\sum_{k=0}^{n-1}\binom{n}{k}f_{2}(k)f_{2}(n-k)+f_{2}(n),
\]
multiply both sides by $x^{n}/n!$, and sum over $n$. We obtain
\[
\sum_{n=0}^{\infty}f_{1}(n+1)\frac{x^{n}}{n!}=\sum_{n=0}^{\infty}\left[\sum_{k=0}^{n-1}\binom{n}{k}f_{2}(k)f_{2}(n-k)\right]\frac{x^{n}}{n!}+\sum_{n=0}^{\infty}f_{2}(n)\frac{x^{n}}{n!}.
\]
We know that $F_{1}^{\prime}(x)=\sum_{n=0}^{\infty}f_{1}(n+1)x^{n}/n!$
and that $F_{2}(x)=\sum_{n=0}^{\infty}f_{2}(n)x^{n}/n!$. Moreover,
\begin{align*}
\sum_{k=0}^{n-1}\binom{n}{k}f_{2}(k)f_{2}(n-k) 
 & =\sum_{k=0}^{n}\binom{n}{k}f_{2}(k)f_{2}(n-k)-f_{2}(n)
\end{align*}
so 
\[
\sum_{n=0}^{\infty}\left[\sum_{k=0}^{n-1}\binom{n}{k}f_{2}(k)f_{2}(n-k)\right]\frac{x^{n}}{n!}=F_{2}^{2}(x)-F_{2}(x).
\]
Hence,
we have 
$
F_{1}^{\prime}(x)=F_{2}^{2}(x).
$
The other three differential equations are obtained via the same method
using the other three recurrences. Then one can verify (e.g., using Maple or another computer algebra system) that the four generating
functions stated in the lemma indeed satisfy the system of differential
equations. 
\end{proof}

We don't know how to find the solution of this system of differential equations without
starting with a guess. We initially found \eqref{e-F2} empirically, and, as described in Section \ref{s-3}, derived \eqref{e-Fgf} from it. We can compute $F_1(x)$ and $F_2(x)$ from \eqref{e-Fgf} by bisection, and then compute $G_1(x)$ and $G_2(x)$ from the differential equations. The system is reminiscent a matrix Riccati equation but is not quite of that form. (Matrix Riccati equations sometimes arise in permutation enumeration problems; see Collins et al.~\cite{MR675349}.) 

Theorem \ref{t-m=00003D3} then follows from the fact that 
\begin{align*}
F(x) & =F_{1}(x)+F_{2}(x)\\
 & =\frac{3\sin\left(\frac{1}{2}x\right)+3\cosh\left(\frac{1}{2}\sqrt{3}x\right)}{3\cos\left(\frac{1}{2}x\right)-\sqrt{3}\sinh\left(\frac{1}{2}\sqrt{3}x\right)},
\end{align*}
which can, again, be verified by Maple.

\subsection{Asymptotics for $f(n)/n!$}

Next, we compute the asymptotics for $f(n)/n!$. We use the following result, which is Theorem IV.7 of \cite{Flajolet2009}. 
\begin{thm}[Exponential Growth Formula]
 Let $F(z)$ be analytic at 0, let $f(n)=\left[z^{n}\right]F(z)$,
and let $R$ be the modulus of a singularity nearest to the origin,
i.e., 
\[
R=\sup\left\{\,r\geq0 \mid \text{$F$ is analytic at all points of }0\leq z<r\,\right\} .
\]
Then for all $\epsilon>0$, there exists $N$ such that for all $n>N$,
we have 
\[
\left|f(n)\right|<\left(\frac{1}{\left|R\right|}+\epsilon\right)^{n}.
\]
Furthermore, for infinitely many $n$ we have 
\[
\left|f(n)\right|>\left(\frac{1}{\left|R\right|}-\epsilon\right)^{n}.
\]
 
\end{thm}
The exponential growth formula implies that the growth rate of
rate of the coefficients of a meromorphic function can be determined from the
location of its poles closest to the origin. We use this idea
to extract the asympotic data from our exponential generating function
$F$. 
\begin{thm}
\label{t-asm}Let $f(n)$ be the number of $n$-permutations with
all valleys even and all peaks odd. Then 
\[
\frac{f(n)}{n!}=2\beta^{n+1}+O\left(\delta^n \right), \mbox{ as }n\rightarrow\infty,
\]
where $\beta=0.7693323708\cdots$ and $\delta=0.3049639861\cdots$.\end{thm}
\begin{proof}
Let $F(z)$ be defined by \eqref{e-Fgf} as a function of a complex variable $z$. 
Let $D(z) ={3\cos(\frac{1}{2}z)-\sqrt{3}\sinh(\frac{1}{2}\sqrt{3}z)}$ be the denominator of $F(z)$. Then $D(z)$ is an entire function, so $F(z)$ is meromorphic on the entire complex plane.
Using Maple,  we find that the zero of $D(z)$ of smallest modulus is a simple zero at  
$z=\alpha:=1.299828316\cdots$  and the zeros with the second
smallest modulus are at $z=-2.058295121\cdots\pm(2.552598427\cdots)i$
with modulus $\gamma:=3.279075713\cdots$. 
Then we can write 
\[
F(z)=\frac{J(z)}{z-\alpha},
\]
where $J(z)$  is analytic in  $|z|<\gamma$. Then 
\[
\frac{f(n)}{n!}=\left[z^{n}\right]\frac{J(\alpha)}{z-\alpha}+\left[z^{n}\right]\left(F(z)-\frac{J(\alpha)}{z-\alpha}\right)
\]
and 
\[
K(z):=F\left(z\right)-\frac{J(\alpha)}{z-\alpha}
\]
is analytic at $z=\alpha$, where $J(\alpha)/(z-\alpha)$ is the principal
part of the Laurent expansion of $F(z)$ about $\alpha$. The singularities
of $K$ closest to the origin have modulus $\gamma$, so by the exponential
growth formula, we obtain 
\[
\left[z^{n}\right]\left(F(z)-\frac{J(\alpha)}{z-\alpha}\right)=O\left(\left(\gamma^{-1}+\epsilon\right)^{n}\right)
\]
as $n\rightarrow\infty$ for any $\epsilon>0$. But since the two singularities of $K(z)$ of smallest modulus are simple poles, we can replace $\gamma^{-1}+\epsilon$ with 
$\gamma^{-1}$. Thus
\begin{align}
\frac{f(n)}{n!} & =\left[z^{n}\right]\frac{J(\alpha)}{z-\alpha}+O\left(\gamma^{-n}\right)\nonumber \\
 & =\left[z^{n}\right]\left(-\frac{J(\alpha)/\alpha}{1-\frac{z}{\alpha}}\right)+O\left(\gamma^{-n}\right)\nonumber \\
 & =\left[z^{n}\right]\sum_{n=0}^{\infty}\left(-\frac{J(\alpha)}{\alpha^{n+1}}\right)z^{n}+O\left(\gamma^{-n}\right)\nonumber \\
 & =-\frac{J(\alpha)}{\alpha^{n+1}}+O\left(\gamma^{-n}\right)\label{e-asm1}
\end{align}
as $n\rightarrow\infty$. It remains to determine
$J(\alpha)$. Let $U(z)=1/F(z)$. We have that 
\[
J(z)=(z-\alpha)F(z),
\]
so 
\begin{align}
J(\alpha) & =\lim_{z\rightarrow\alpha}(z-\alpha)F(z)\nonumber \\
 & =\lim_{z\rightarrow\alpha}\frac{z-\alpha}{U(z)}\nonumber \\
 & =\frac{1}{U^{\prime}(\alpha)}\label{e-asm2}
\end{align}
by L'H\^opital's rule. Moreover, 
\[
U^{\prime}(z)=-\frac{3}{2}\cdot\frac{\left[3\cos\left(\frac{1}{2}z\right)-\sqrt{3}\sinh\left(\frac{1}{2}\sqrt{3}z\right)\right]\left[\cos\left(\frac{1}{2}z\right)+\sqrt{3}\sinh\left(\frac{1}{2}\sqrt{3}z\right)\right]}{\left[3\sin\left(\frac{1}{2}z\right)+3\cosh\left(\frac{1}{2}\sqrt{3}z\right)\right]^{2}}-\frac{1}{2},
\]
and since $3\cos\left(\frac{1}{2}\alpha\right)-\sqrt{3}\sinh\left(\frac{1}{2}\sqrt{3}\alpha\right)=0$,
it follows that 
\[
U^{\prime}(\alpha)=-\frac{1}{2}.
\]
Therefore, from (\ref{e-asm1}) and (\ref{e-asm2}) we obtain 
\[
\frac{f(n)}{n!}=\frac{2}{\alpha^{n+1}}+O\left(\gamma^{-n}\right),
\]
and applying the substitutions $\beta:=\alpha^{-1}$ and $\delta:=\gamma^{-1}$
completes the result.
\end{proof}

By including more poles of $F(z)$, we could get better approximations.

\section{Alternating Descents and Runs}
\label{s-3}

\subsection{The Euler number formula}

Next we show that $F(x)$, as given in \eqref{e-Fgf}, is equal to  
\begin{multline*}
\qquad\left(1-x+E_{3}\frac{x^{3}}{3!}-E_{4}\frac{x^{4}}{4!}+E_{6}\frac{x^{6}}{6!}-E_{7}\frac{x^{7}}{7!}+\cdots\right)^{-1}  =\\
\left[\sum_{n=0}^{\infty}\left(E_{3n}\frac{x^{3n}}{\left(3n\right)!}-E_{3n+1}\frac{x^{3n+1}}{\left(3n+1\right)!}\right)\right]^{-1},\qquad
\end{multline*}
where the Euler numbers $E_n$ are defined by $\sum_{n=0}^\infty E_n x^n\!/n! = \sec x +\tan x$.

Let $E(x)=\sec x+\tan x$. Then by multisection, we have
\[
\sum_{n=0}^{\infty}E_{3n}\frac{x^{3n}}{\left(3n\right)!}=\frac{1}{3}\left(E(x)+E(\omega x)+E(\omega^{2}x)\right)
\]
and 
\[
\sum_{n=0}^{\infty}E_{3n+1}\frac{x^{3n+1}}{\left(3n+1\right)!}=\frac{1}{3}\left(E(x)+\omega^{-1}E(\omega x)+\omega^{-2}E(\omega^{2}x)\right),
\]
where $\omega$ is the primitive cube root of unity $e^{2\pi i/3}$.
It can then be verified, for example by using Maple, that 
\begin{align}
F(x) & =\frac{3\sin\left(\frac{1}{2}x\right)+3\cosh\left(\frac{1}{2}\sqrt{3}x\right)}{3\cos\left(\frac{1}{2}x\right)-\sqrt{3}\sinh\left(\frac{1}{2}\sqrt{3}x\right)}\nonumber \\
 & =\left[\frac{1}{3}\left(E(x)+E(\omega x)+E(\omega^{2}x)\right)-\frac{1}{3}\left(E(x)+\omega^{-1}E(\omega x)+\omega^{-2}E(\omega^{2}x)\right)\right]^{-1}\nonumber \\
 & =\left[\sum_{n=0}^{\infty}\left(E_{3n}\frac{x^{3n}}{\left(3n\right)!}-E_{3n+1}\frac{x^{3n+1}}{\left(3n+1\right)!}\right)\right]^{-1}.\label{e-FE}
\end{align}

Thus we see that \eqref{e-FE} is the exponential generating function for permutations 
with all valleys even and all peaks odd. 

In the next sections, we give a more direct proof of this fact.

\subsection{Permutations with prescribed descent sets and alternating descent sets}
First we discuss a connection between counting permutations by descent set and counting permutations by alternating descent set. It is often convenient to represent descent sets of permutations of $[n]$ by compositions of $n$, i.e., sequences of positive integers with sum $n$. To a subset $S\subseteq [n-1]$ with elements $s_1<s_2<\cdots <s_k$ we  associate  the composition $C(S) = (s_1, s_2-s_1,\dots, s_j-s_{j-1}, n-s_j)$ of $n$. We define the mapping $D$ from compositions of $n$ to subsets of $[n-1]$  given by the formula $D(L) = \{L_1, L_1+L_2,\dots,L_1+\cdots+L_{k-1}\}$, where $L=(L_1,L_2,\dots, L_k)$. Then $C$ and $D$ are inverse bijections.

We partially order the compositions of $n$ by reverse refinement,
so if $L=\left(L_{1},\dots,L_{k}\right)$ then $L$ covers $M$ if
and only if $M$ can be obtained from $L$ by replacing two consecutive
parts $L_{i},L_{i+1}$ with $L_{i}+L_{i+1}$. Thus, for example, $\left(7,6\right)<\left(1,2,4,5,1\right)$.
Then $C$ and $D$ are order-preserving bijections.

If a permutation $\pi$ in $\S_n$ has descent set $S$, then we call $C(S)$ the \emph{descent composition} of $\pi$, and we define the \emph{alternating descent composition} of $\pi$ analogously.

If $L=(L_1,\dots, L_k)$ is a composition of $n$ then we write $\binom nL$ for the multinomial coefficient $\binom{n}{L_1,\dots, L_k}$, and we write $\binomE nL$ for $\binom{n}{L}E_{L_1}\cdots E_{L_k}$.

\begin{lem} 
\label{l-7}
Let $L$ be a composition of $n$. Then the number of permutations of $[n]$ with descent set contained in $D(L)$ is the multinomial coefficient $\binom{n}{L}$ and the number of  permutations of $[n]$ with alternating descent set contained in  $D(L)$  is
$\binomE{n}{L}$
\end{lem}

\begin{proof}
Let $L=(L_1,\dots, L_k)$.
To prove the first formula, we construct a bijection from the set of ordered partitions with block sizes $L_1,\dots, L_k$ to the set of permutations of $[n]$ with descent set contained in $D(L)$.
Given an ordered partition $(B_1,\dots, B_k)$ of $[n]$ into blocks of sizes $L_1,\dots, L_k$, 
let $B_i\up$ be the word obtained by arranging the elements of $B_i$ in increasing order. Then the map that takes 
$(B_1,\dots,B_k)$ to the concatenation $B_1\up\cdots B_k\up$ is easily seen to be the required bijection.

For the second formula, we proceed in the same way, but instead of arranging the elements of each block in increasing order, we arrange them into either an alternating or reverse-alternating permutation, according to the parity of the starting position, so as to create a segment of the permutation with no alternating descents.
\end{proof}

We denote by $l(L)$ the number of parts of the composition $L$.

\begin{lem}
For $L\vDash n$, let $\beta(L)$ be the number of permutations of $[n]$ with descent composition $L$ and let $\hat\beta(L)$ be the number of permutations of $[n]$ with alternating descent composition $L$. Then 
\begin{align}
\label{e-beta} 
\beta(L) = \sum_{K\le L} (-1)^{l(L) - l(K)}\binom{n}K\\
\shortintertext{and}
\label{e-betahat}
\hat\beta(L) = \sum_{K\le L} (-1)^{l(L) - l(K)}\binomE{n}K\
\end{align}
\end{lem}
\begin{proof}
By Lemma \ref{l-7}, $\binom{n}{L} = \sum_{K\le L }\beta(K)$ and $\binomE{n}{L} = \sum_{K\le L}\hat\beta(K)$. The results then follow by inclusion-exclusion.
\end{proof}
Formula \eqref{e-beta}  is due to MacMahon \cite[Volume 1, p.~190]{macmahon} and formula \eqref{e-betahat} is due to Chebikin \cite{chebikin}. 

Note that $\binom{n}{L}$ is the coefficient of $x^n\!/n!$ in 
\begin{equation*}
\frac{x^{L_1}}{L_1!}\cdots \frac{x^{L_k}}{L_k!}
\end{equation*}
and  $\binomE{n}{L}$ is the coefficient of $x^n\!/n!$ in 
\begin{equation*}
E_{L_1}\frac{x^{L_1}}{L_1!}\cdots E_{L_k} \frac{x^{L_k}}{L_k!}.
\end{equation*}

This suggests that if we have an exponential generating function for counting permutations weighted in some way by their descent sets  then we might find an analogous formula for alternating descent sets by replacing $x^n\!/n!$ with $E_n x^n\!/n!$ everywhere. We will make this idea precise in the next section by using noncommutative symmetric functions.

\section{Noncommutative Symmetric Functions}

\subsection{Definitions}

In this section we give an introduction to some of elements of the theory of noncommutative symmetric functions, which has been extensively developed in \cite{ncsf1} and many further publications. We note that our definitions and notation are different from those used in~\cite{ncsf1}, but the algebras are the same.

Although our main results could be proved using ordinary commutative symmetric functions,  noncommutative symmetric functions provide a more natural setting.

Let $F$ be a field of characteristic zero. We work in the algebra $F\langle\langle X_1, X_2,\dots \rangle\rangle$ of formal power series in noncommuting variables $X_1, X_2,\dots$, grading by degree in the $X_i$.
Let 
\[
\mathbf{h}_{n}=\sum_{i_{1}\leq\cdots\leq i_{n}}X_{i_{1}}X_{i_{2}}\cdots X_{i_{n}}.
\]
These
are noncommutative versions of the complete  symmetric
functions $h_{n}$ (see \cite[Chapter 7]{Stanley2001}). For any composition $L$, let $\h_L = \h_{L_1}\cdots \h_{L_k}$.

The algebra $\Sym$ of \emph{noncommutative symmetric functions} with  coefficients in $F$ is the subalgebra of $F\langle\langle X_1, X_2,\dots \rangle\rangle$ consisting of all (possibly infinite) linear combinations of the $\h_L$ with coefficients in $F$.
We denote by $\Sym_n$ the vector space of noncommutative symmetric functions homogeneous of degree $n$, so $\Sym_n$ is the span of  the set $\{\h_L\}_{L\vDash n}$, where $L\vDash n$ means that $L$ is a  composition of~$n$.
We define the noncommutative symmetric functions $\r_L$ by 
\[
\mathbf{r}_{L}=\sum_{L}X_{i_{1}}X_{i_{2}}\cdots X_{i_{n}}
\]
where the sum is over all $\left(i_{1},\dots,i_{n}\right)$  satisfying
\begin{equation}
\label{e-descomp}
\underset{L_{1}}{\underbrace{i_{1}\leq\cdots\leq i_{L_{1}}}}>\underset{L_{2}}{\underbrace{i_{L_{1}+1}\leq\cdots\leq i_{L_{1}+L_{2}}}}>\cdots>\underset{L_{k}}{\underbrace{i_{L_{1}+\cdots+L_{k-1}+1}\leq\cdots\leq i_{n}}}.
\end{equation}

It is clear that 
\begin{equation}
\label{e-htor}
\mathbf{h}_{L}=\sum_{K\leq L}\mathbf{r}_{K},
\end{equation}
so by inclusion-exclusion, 
\begin{equation}
\mathbf{r}_{L}=\sum_{K\leq L}(-1)^{l(L)-l(K)}\mathbf{h}_{K},\label{e-r}
\end{equation}
The set $\left\{ \mathbf{r}_{L}\right\} _{L\vDash n}$  is linearly independent,
since for $L\neq M$, $\mathbf{r}_{L}$ and $\mathbf{r}_{M}$ have
no terms in common. It follows from (\ref{e-htor}) that $\left\{ \mathbf{r}_{L}\right\} _{L\vDash n}$ 
spans $\Sym _n$, so $\left\{ \mathbf{r}_{L}\right\} _{L\vDash n}$
and  $\left\{ \mathbf{h}_{L}\right\} _{L\vDash n}$
are both bases for $\Sym_n$.

\subsection{Homomorphisms}
We define a homomorphism $\Phi$ from $\Sym$ to $F[[x]]$ by $\Phi(\h_n) = x^n\!/n!$. 
Then if $L$ is a composition of $n$ we have
\begin{equation*}
\Phi(\h_L) = \frac{x^{L_1}}{L_1!}\cdots\frac{x^{L_k}}{L_k!} = \binom{n}{L} \frac{x^n}{n!}.
\end{equation*}

Another description of $\Phi$, which we will not need for our proofs, but which motivates the definition, is that if $f$ is a noncommutative symmetric function then the coefficient of $x^n\!/n!$ in $\Phi(f)$ is the coefficient of $x_1 x_2\cdots x_n$ in the result of replacing $X_1,X_2,\dots$ with commuting variables $x_1, x_2, \dots$.

\begin{lem}
\label{l-Phi}
If $L\vDash n$ then $\Phi(\r_L) = \beta(L)\, x^n\!/n!$.
\end{lem}

\begin{proof}
We have
\begin{align*}
\Phi(\r_L) &= \Phi\biggl(\sum_{K\leq L}(-1)^{l(L)-l(K)}\mathbf{h}_{K}\biggr), \text{ by \eqref{e-r},} \\
  &=\sum_{K\leq L}(-1)^{l(L)-l(K)}\Phi(\mathbf{h}_{K})\\
  &=\sum_{K\leq L}(-1)^{l(L)-l(K)}\binom{n}{K}\frac{x^n}{n!}\\
  &=\beta(L)\frac{x^n}{n!}, \text{ by \eqref{e-beta}}.
\end{align*}
\end{proof}

We define another homomorphism $\hat\Phi: \Sym\to F[[x]]$ by $\hat\Phi(\h_n)=E_n x^n\!/n!$. Then the following formula for $\hat\Phi(\r_n)$ is proved in exactly the same way as Lemma \ref{l-Phi}:

\begin{lem}
\label{l-Phihat}
If $L\vDash n$ then $\hat\Phi(\r_L) = \hat\beta(L)\, x^n\!/n!$.\qed
\end{lem}

Lemmas \ref{l-Phi} and \ref{l-Phihat} explain the connection between generating functions for descents and for alternating descents: If an exponential generating function for counting permutations by descents can be obtained by applying $\Phi$ to a noncommutative symmetric function, then applying $\hat\Phi$ will give an analogous generating function for counting permutations by alternating descents. 

\subsection{Counting words by runs}
The next result gives a very general noncommutative symmetric generating function to which we can apply $\Phi$ and $\hat\Phi$.
It is Theorem 5.2 of Gessel \cite{gessel-thesis} and is also a
noncommutative version of a special case of Theorem 4.2.3 of Goulden
and Jackson \cite{MR702512}. See also Jackson and Aleliunas \cite[Theorem 4.1]{MR0450080}.
\begin{thm}
\label{t-runs}Let $w_{1},w_{2},\dots$ be arbitrary commuting weights
and define $a_{0}=1,a_{1},a_{2},\dots$ by 
\begin{equation}
\sum_{n=0}^{\infty}a_{n}z^{n}=\biggl(\sum_{n=0}^{\infty}w_{n}z^{n}\biggr)^{-1},\label{e-cx}
\end{equation}
where $w_{0}=1$. Then 
\[
\sum_{L}w_{L}\mathbf{r}_{L}=\biggl(\sum_{n=0}^{\infty}a_{n}\mathbf{h}_{n}\biggr)^{-1}
\]
where the sum on the left is over all compositions $L$, and $w_{L}=w_{L_{1}}\cdots w_{L_{k}}$ where $L=\left(L_{1},\dots,L_{k}\right)$.
\end{thm}
\begin{proof}
Let us set $a_{n}=-u_{n}$ for $n>0$, and for a composition $K=\left(K_{1},\dots,K_{k}\right)$
 let $u_{K}=u_{K_{1}}\cdots u_{K_{k}}$.
Then 
\begin{align}
\biggl(\sum_{n=0}^{\infty}a_{n}\mathbf{h}_{n}\biggr)^{-1} 
 & =\biggl(1-\sum_{n=1}^{\infty}u_{n}\mathbf{h}_{n}\biggr)^{-1}\nonumber \\
 & =\sum_{K}u_{K}\mathbf{h}_{K}\nonumber \\
 & =\sum_{K}u_{K}\sum_{L\leq K}\mathbf{r}_{L}\nonumber \\
 & =\sum_{L}\mathbf{r}_{L}\sum_{K\geq L}u_{K}.\label{e-rc}
\end{align}
By (\ref{e-cx}) we have 
\[
\sum_{n=0}^{\infty}w_{n}z^{n}=\biggl(1-\sum_{n=1}^{\infty}u_{n}z^{n}\biggr)^{-1}
\]
so 
\begin{equation}
w_{n}=\sum_{K\vDash n}u_{K}.\label{e-sc}
\end{equation}
From (\ref{e-sc}) we see that 
\begin{equation}
w_{L}=\sum_{K\geq L}u_{K};\label{e-sK}
\end{equation}
then the theorem follows from (\ref{e-rc}) and (\ref{e-sK}). \end{proof}
\begin{cor}
\label{c-runs} Let $m$ be a positive integer. Then
\begin{align*}
\sum_{L}\mathbf{r}_{L} & =\biggl(1-\mathbf{h}_{1}+\mathbf{h}_{m}-\mathbf{h}_{m+1}+\cdots\biggr)^{-1}\\
 & =\biggl(\sum_{n=0}^{\infty}\left(\mathbf{h}_{mn}-\mathbf{h}_{mn+1}\right)\biggr)^{-1}
\end{align*}
where the sum on the left is over all compositions $L$ with all parts
less than $m$.\end{cor}
\begin{proof}
We apply Theorem \ref{t-runs} with $w_{i}=1$ for $i<m$ and $w_{i}=0$
for $i\geq m$. We have 
\[
\sum_{n=0}^{\infty}w_{n}z^{n}=\frac{1-z^{m}}{1-z},
\]
so 
\[
\sum_{n=0}^{\infty}a_{n}z^{n}=\frac{1-z}{1-z^{m}}=\sum_{n=0}^{\infty}\left(z^{mn}-z^{mn+1}\right).
\]
Then by Theorem \ref{t-runs}, 
\[
\sum_{L}\mathbf{r}_{L}=\biggl(\sum_{n=0}^{\infty}\left(\mathbf{h}_{mn}-\mathbf{h}_{mn+1}\right)\biggr)^{-1},
\]
where the sum on the left is over all compositions $L$ with all parts
less than $m$.
\end{proof}

Applying $\Phi$ to Corollary \ref{c-runs} gives David and Barton's result \eqref{e-F2}. Applying $\hat\Phi$ to Corollary \ref{c-runs} gives 
\begin{equation*}
\left[\sum_{n=0}^{\infty}\left(E_{mn}\frac{x^{mn}}{\left(mn\right)!}-E_{mn+1}\frac{x^{mn+1}}{\left(mn+1\right)!}\right)\right]^{-1}
\end{equation*}
as the exponential generating function for permutations in which every alternating run has length less than $m$; the case $m=3$, as noted earlier, is equivalent to \eqref{e-FE}.

There are two additional special cases of  Theorem \ref{t-runs} that  
are also of particular interest:
 If we set $w_{m}=1$ and $w_{n}=0$ for $n\neq m$ then we find that
\begin{equation}
\sum_{n=0}^{\infty}\mathbf{r}_{\left(m^{n}\right)}
   =\biggl(\sum_{n=0}^{\infty}\left(-1\right)^{n}\mathbf{h}_{mn}\biggr)^{-1},\label{e-h=m}
\end{equation}
where $\left(m^{n}\right)$ is the composition $(\underset{n}{\underbrace{m,m,\dots,m}})$.
The case $m=2$ of \eqref{e-h=m} is equation (98) in Proposition
5.23 of \cite{ncsf1}. Applying $\Phi$ to \eqref{e-h=m} gives Carlitz's result \cite{Carlitz73} that 
\[\biggl(\sum_{n=0}^\infty (-1)^n \frac{x^{mn}}{(mn)!}\biggr)^{-1}\] is the exponential generating function for 
permutations in which every increasing run has length $m$. Applying $\hat\Phi$ to \eqref{e-h=m}, we see that 
\begin{equation}
\label{e-19}
\biggl(\sum_{n=0}^{\infty}E_{mn}\frac{x^{mn}}{(mn)!}\biggr)^{-1}
\end{equation}
counts permutations in which every alternating run has length $m$. If $m$ is 3 or 4, this result can be stated in a simpler way. For $m=4$, \eqref{e-19} counts permutations of $[4n]$ with descent set $\{2,6,10,\dots, 4n-2\}$. For $m=3$, \eqref{e-19} counts permutations whose sequence of ascents (denoted $U$) and descents (denoted $D$) is of the form
$UD^3U^3D^3\cdots D^3U$ or $UD^3U^3D^3\cdots U^3D$.

If we set $w_{n}=t$ for all $n\geq1$  in \eqref{e-cx} then 
\begin{align*}
\sum_{n=0}^{\infty}a_{n}z^{n} & =\biggl(1+\sum_{n=1}^{\infty}tz^{n}\biggr)^{-1}\\
 & =\biggl(1+\frac{tz}{1-z}\biggr)^{-1}\\
 & =1-\frac{tz}{1-\left(1-t\right)z}\\
 & =1-\sum_{n=1}^{\infty}t\left(1-t\right)^{n-1}z^{n}.
\end{align*}
Thus by Theorem \ref{t-runs} we have 
\begin{align}
\sum_{L}t^{l(L)}\mathbf{r}_{L} & =\biggl[1-\sum_{n=1}^{\infty}t\left(1-t\right)^{n-1}\mathbf{h}_{n}\biggr]^{-1}\nonumber \\
 & =\left(1-t\right)\biggl[1-t\sum_{n=0}^{\infty}\left(1-t\right)^{n}\mathbf{h}_{n}\biggr]^{-1},\label{e-heuler}
\end{align}
where the first sum is over all compositions $L$. Formula (\ref{e-heuler})
is Proposition 5.20 of \cite{ncsf1}.
Applying $\Phi$ gives the well-known generating function for the Eulerian polynomials,
\begin{equation*}
1+\sum_{n=1}^\infty A_n(t)\frac{x^n}{n!} = \frac{1-t}{1-te^{(1-t)x}},
\end{equation*}
where 
\begin{equation*}
A_n(t) = \sum_{\pi\in \S_n} t^{\des(\pi) + 1},
\end{equation*}
and $\des(\pi)$ is the number of descents of $\pi$.
Applying $\hat\Phi$ gives
\begin{equation}
\label{e-altdes}
1+\sum_{n=1}^\infty \hat A_n(t)\frac{x^n}{n!} = \frac{1-t}{1-t\bigl(\sec(1-t)x + \tan (1-t)x\bigr)},
\end{equation}
where 
\begin{equation*}
\hat A_n(t) = \sum_{\pi\in \S_n} t^{\altdes(\pi) + 1}
\end{equation*}
and  $\altdes(\pi)$ is the number of alternating descents of $\pi$.
Equation \eqref{e-altdes} is equivalent to Theorem 4.2 of Chebikin \cite{chebikin}.

\subsection{Counting words by runs with distinguished last run}
There is a generalization of Theorem \ref{t-runs} in which the last run is weighted differently from the other runs.
(See Goulden and Jackson \cite[Theorem 4.2.19]{MR702512}, Jackson and Aleliunas \cite[Theorem 11.1]{MR0450080} and Gessel \cite[Theorem 6.12(b)]{gessel-thesis}.) We omit the proof, which is similar to the proof of Theorem \ref{t-runs}. There is a further generalization in which both the first and last runs are weighted differently that we do not state here. (See  \cite[Theorem 11.2]{MR0450080} and \cite[Theorem 6.12(c)]{gessel-thesis}.)

\begin{thm}
\label{t-lrun}Let $w_{1},w_{2},\dots$  and $v_1,v_2,\dots$ be arbitrary commuting weights
and define $a_{0}=1,a_{1},a_{2},\dots$ and $b_1, b_2,\dots$ by 
\begin{equation*}
\sum_{n=0}^{\infty}a_{n}z^{n}=\biggl(\sum_{n=0}^{\infty}w_{n}z^{n}\biggr)^{-1},
\end{equation*}
where $w_{0}=1$, and
\begin{equation*}
\sum_{n=1}^\infty b_n z^n =\sum_{n=1}^\infty v_n z^n\!\!\biggm/\sum_{n=0}^{\infty}w_{n}z^{n}.
\end{equation*}
Then
\[
\sum_{L}w_{L_1}\cdots w_{L_{k-1}}v_{L_k}\mathbf{r}_{L}
=\biggl(\sum_{n=0}^{\infty}a_{n}\mathbf{h}_{n}\biggr)^{-1}
\sum_{n=1}^\infty b_n\mathbf{h}_{n}
\]
where the sum on the left is over all nonempty compositions $L=\left(L_{1},\dots,L_{k}\right)$.
\end{thm}

We use Theorem \ref{t-lrun} to find generating functions involving the Euler numbers for the numbers $g(n)$, discussed in Section \ref{s-2}, which count permutations of $[n]$ with every alternating run of length less than 3 that end with an ascent. Let $c(n)$ be the number of permutations of $[n]$ with every alternating run of length less than 3 in which the last alternating run has length 1, and let $d(n)$ be the number of permutations of $[n]$ with every alternating run of length less than 3 in which the last alternating run has length 2. Thus $f(n) = c(n) + d(n)$ for $n>0$.
It is not difficult to see that for $n>0$,
\begin{equation*}
g(n)=\begin{cases}
  c(n),&\text{if $n$ is odd,}\\
  d(n),&\text{if $n$ is even,}
  \end{cases}
\end{equation*}
and thus 
\begin{equation*}
f(n) -g(n)=\begin{cases}
  c(n),&\text{if $n$ is even,}\\
  d(n),&\text{if $n$ is odd.}
  \end{cases}
\end{equation*}

Taking Theorem  \ref{t-lrun} with $\sum_{n\ge0} w_n z^n = 1+z+z^2$ and $\sum_{n\ge1}v_n = x$ or $x^2$, and then applying the homomorphism $\hat \Phi$, gives the following result.

\begin{thm}
The exponential generating functions for $c(n)$ and $d(n)$ are 
\begin{multline*}
\sum_{n=0}^\infty c(n) \frac{x^n}{n!} \\
  =
  \sum_{n=0}^{\infty}\left(E_{3n+1}\frac{x^{3n+1}}{\left(3n+1\right)!}-E_{3n+2}\frac{x^{3n+2}}{\left(3n+2\right)!}\right)
  \biggm/
  \sum_{n=0}^{\infty}\left(E_{3n}\frac{x^{3n}}{\left(3n\right)!}-E_{3n+1}\frac{x^{3n+1}}{\left(3n+1\right)!}\right)
\end{multline*}
and
\begin{multline*}
\sum_{n=0}^\infty d(n) \frac{x^n}{n!} \\
  =
  \sum_{n=0}^{\infty}\left(E_{3n+2}\frac{x^{3n+2}}{\left(3n+2\right)!}-E_{3n+3}\frac{x^{3n+3}}{\left(3n+3\right)!}\right)
  \biggm/
  \sum_{n=0}^{\infty}\left(E_{3n}\frac{x^{3n}}{\left(3n\right)!}-E_{3n+1}\frac{x^{3n+1}}{\left(3n+1\right)!}\right).
\end{multline*}
\end{thm}

The first few values of $c(n)$ and $d(n)$ are as follows:

\[\vbox{\halign{\ \hfil\strut$#$\hfil\ \vrule&&\hfil\ $#$\ \hfil\cr
n&0&1&2&3&4&5&6&7&8&9&10&11&12\cr
\noalign{\hrule}
c(n)&0&1&1&3&9&34&159&853&5249&36369&279711&2367212&21854625\cr
d(n)&0&0&1&1&4&16&70&385&2365&16337&125870&1064810&9829820\cr}}
\]

The generating functions for $c(n)$ and $d(n)$ can also be written in terms of trigonometric and hyperbolic functions:
\begin{gather*}
\sum_{n=0}^\infty c(n) \frac{x^n}{n!} 
  =\frac{2\sqrt3 \sinh(\frac12 \sqrt3 x)}{3\cos(\frac 12 x) -\sqrt3 \sinh(\frac12 \sqrt 3 x)}\\
  \shortintertext{and}
\sum_{n=0}^\infty d(n) \frac{x^n}{n!} 
  =\frac{3\sin(\frac12 x) -3\cos(\frac12 x) + 3\cosh(\frac12\sqrt3 x) -\sqrt3\sinh(\frac12\sqrt3 x)}{3\cos(\frac 12 x) -\sqrt3 \sinh(\frac12 \sqrt 3 x)}.
\end{gather*}

%As another example, applying Theorem \ref{t-lrun} to count permutations

\section{The Alternating Major Index}
The \emph{major index} $\maj(\pi)$ of a permutation $\pi$ is the sum of the  descents of $\pi$, so for example, the major index of $15243$ is $2+4=6$. Following Remmel \cite{remmel}, we define the \emph{alternating major index} $\altmaj(\pi)$ of $\pi$ to be the sum of the alternating descents of $\pi$.
Remmel proved the following result, which reduces to a formula equivalent to \eqref{e-altdes} for $q=1$.
\begin{thm}
\label{t-remmel}
\begin{equation}
\label{e-remmel}
\sum_{n=0}^\infty \frac{x^n}{n!} \frac{\sum_{\pi\in \S_n} t^{\altdes(\pi)} q^{\altmaj(\pi)}}{(1-t)(1-tq)\cdots (1-tq^n)}
  =\sum_{k=0}^\infty t^k \prod_{j=0}^k \bigl(\sec(xq^j) +\tan(xq^j)\bigr)
\end{equation}
\end{thm}

We note multiplying both sides of \eqref{e-remmel} by $1-t$ and taking the limit as $t\to 1$ gives
\begin{equation}
\label{e-altmaj}
\sum_{n=0}^\infty \frac{x^n}{n!} \frac{\sum_{\pi\in \S_n} q^{\altmaj(\pi)}}
{(1-q)\cdots (1-q^n)}
  =\prod_{j=0}^\infty \bigl(\sec(xq^j) +\tan(xq^j)\bigr).
\end{equation}

We will now derive \eqref{e-remmel} and \eqref{e-altmaj} from formulas for noncommutative symmetric functions. For any composition $L=(L_1,\dots, L_k)$, let us define the major index of $L$ to be $\maj(L) = (k-1)L_1 + (k-2)L_2+\cdots + L_{k-1}$. 
Then if $\pi$ is a permutation with descent composition $L$, we have $\maj(\pi) = \maj(L)$.

Let $H(u) = \sum_{n=0}^\infty \h_n u^n$. (So $\Phi(H(u))= e^{ux}$ and $\hat\Phi(H(u)) = \sec ux + \tan ux$.) 
Following \cite{ncsf1}, we use the notation 
%$\overset{\leftarrow}{\prod}_{a\le j \le b}T_j$ 
\[\leftprod{a\le j \le b}T_j\]
to denote the noncommutative product $T_bT_{b-1}\cdots T_a$.
Then the following result is Proposition 5.10 of \cite{ncsf1}:
\begin{lem}
\label{l-ncmaj}
\begin{equation*}
\leftprod{j\ge0}H(q^j) = \sum_{n=0}^\infty \frac{\sum_{L\vDash n}q^{\maj(L)} \r_L}{(1-q)(1-q^2)\cdots(1-q^n)}
\end{equation*}
\end{lem}
Formula \eqref{e-altmaj} follows directly from Lemma \ref{l-ncmaj} on applying $\hat\Phi$.
Similarly,  \eqref{e-remmel} is obtained by applying $\hat\Phi$ to the following result, which is essentially Theorem 8.2 of Gessel \cite{gessel-thesis}.
\begin{lem}

If $L$ is a nonempty composition, let $\des(L)= l(L) -1$, and if $L$ is the empty composition, let $\des(L) = 0$. Then 
\begin{equation}
\label{e-desmaj}
\sum_{k=0}^\infty t^k \leftprod{0\le j\le k} H(q^j) = \sum_{n=0}^\infty  \frac{\sum_{L\vDash n} t^{\des(L)}q^{\maj(L)}\r_L}{(1-t)\cdots(1-tq^n)}.
\end{equation}
\end{lem}

\begin{proof}
Let us define a \emph{barred word} to be a word in the letters $X_1, X_2,\dots$ and the symbol $|$ (the bar) with the property that $X_i$ is never immediately followed by any $X_j$ with $j<i$. So, for example, $\Bar X_1 X_3 \Bar X_2 X_2 \Bar X_4 \,||\, X_3$ is a barred word. For a barred word $B$, let $e(B)$ be the number of pairs consisting of a letter $X_i$ in $B$ and a bar to its right. Thus  $e(X_1 \Bar X_3 \Bar)=3$ since in the barred word $X_1 \Bar X_3 \Bar$,  there are two bars to the right of $X_1$ and one bar to the right of  $X_3$. Let $b(B)$ be the number of bars in $B$ and let $u(B)$ be the word in $X_1,X_2,\dots$ that remains when the bars are removed from $B$. 

We will show that both sides of \eqref{e-desmaj} are equal to the sum of $t^{b(B)}q^{e(B)}u(B)$ over all barred permutations $B$.
For the left side, suppose that $B$ is a barred permutation with $k$ bars, so $B = w_k\Bar w_{k-1} \Bar\cdots \Bar w_0$, where each $w_j$ is a word of the form $X_{i_1}\cdots X_{i_m}$, with $i_1\le \cdots\le i_m$. If  each $w_j$ has length $m_j$, then $e(B) = m_1 + 2m_2 +\cdots+km_k$, so the sum of $q^{e(B)}u(B)$ over all barred words with $k$ bars is $H(q^k)H(q^{k-1})\cdots H(1)$.

For the right side, let $w$ be a word in $X_1, X_2,\dots$ with descent composition $L$, so  $w$ is a term in $\r_L$. We think of $w$ as having $n+1$ ``spaces'': before the first letter, between adjacent letters, and after the last letter. We number the spaces 0 to $n$ from left to right. We can obtain a barred word $B$ with $u(B)=w$ by first inserting a bar in each descent space of $w$ and then inserting any number of bars in each of the $n+1$ spaces. A bar in space $i$ contributes a factor of $tq^i$ to $t^{b(B)}q^{e(B)}$. So inserting one bar in each descent space gives a factor of $t^{\des(L)}q^{\maj(L)}$, and then inserting any number of bars in each of the spaces contributes a factor of $1/(1-t)\cdots (1-tq^n)$.
\end{proof}

\bibliographystyle{amsplain}
\addcontentsline{toc}{section}{\refname}\bibliography{new-bibliography}

\providecommand{\bysame}{\leavevmode\hbox to3em{\hrulefill}\thinspace}
\providecommand{\MR}{\relax\ifhmode\unskip\space\fi MR }
% \MRhref is called by the amsart/book/proc definition of \MR.
\providecommand{\MRhref}[2]{%
  \href{http://www.ams.org/mathscinet-getitem?mr=#1}{#2}
}
\providecommand{\href}[2]{#2}
\begin{thebibliography}{10}

\bibitem{Bona2012}
Miklós Bóna, \emph{Combinatorics of {P}ermutations}, 2nd ed., Discrete
  Mathematics and its Applications, CRC Press, 2012.

\bibitem{Carlitz73}
L.~Carlitz, \emph{Permutations with prescribed pattern}, Math. Nachr.
  \textbf{58} (1973), 31--53. \MR{0329912 (48 \#8252)}

\bibitem{chebikin}
Denis Chebikin, \emph{Variations on descents and inversions in permutations},
  Electron. J. Combin. \textbf{15} (2008), no.~1, Research Paper 132, 34.
  \MR{2448882 (2009g:05003)}

\bibitem{MR675349}
C.~B. Collins, I.~P. Goulden, D.~M. Jackson, and O.~M. Nierstrasz, \emph{A
  combinatorial application of matrix {R}iccati equations and their
  {$q$}-analogue}, Discrete Math. \textbf{36} (1981), no.~2, 139--153.
  \MR{675349 (84g:05009)}

\bibitem{David1962}
F.~N. David and D.E. Barton, \emph{Combinatorial {C}hance}, Lubrecht \& Cramer
  Ltd, 1962.

\bibitem{Flajolet2009}
Philippe Flajolet and Robert Sedgewick, \emph{Analytic {C}ombinatorics},
  Cambridge University Press, 2009.

\bibitem{ncsf1}
Israel~M. Gelfand, Daniel Krob, Alain Lascoux, Bernard Leclerc, Vladimir~S.
  Retakh, and Jean-Yves Thibon, \emph{Noncommutative symmetric functions}, Adv.
  Math. \textbf{112} (1995), no.~2, 218--348. \MR{1327096 (96e:05175)}

\bibitem{gessel-thesis}
Ira~Martin Gessel, \emph{Generating {F}unctions and {E}numeration of
  {S}equences}, Ph.D. thesis, Massachusetts Institute of Technology, 1977.

\bibitem{MR702512}
I.~P. Goulden and D.~M. Jackson, \emph{{C}ombinatorial {E}numeration}, John
  Wiley \& Sons, Inc., New York, 1983. \MR{702512 (84m:05002)}

\bibitem{MR0450080}
D.~M. Jackson and R.~Aleliunas, \emph{Decomposition based generating functions
  for sequences}, Canad. J. Math. \textbf{29} (1977), no.~5, 971--1009.
  \MR{0450080 (56 \#8379)}

\bibitem{macmahon}
Percy~A. MacMahon, \emph{Combinatory {A}nalysis}, Two volumes (bound as one),
  Chelsea Publishing Co., New York, 1960, Originally published in two volumes
  by Cambridge University Press, 1915--1916.

\bibitem{Nicolaescu2012}
Liviu Nicolaescu, \emph{Combinatorial morse functions and random permutations},
  \url{http://mathoverflow.net/questions/86193}, 2012.

\bibitem{remmel}
Jeffrey~B. Remmel, \emph{Generating functions for alternating descents and
  alternating major index}, Ann. Comb. \textbf{16} (2012), no.~3, 625--650.
  \MR{2960023}

\bibitem{oeisA059427}
N.~J.~A. Sloane, \emph{The {O}n-{L}ine {E}ncyclopedia of {I}nteger
  {S}equences}, published electronically at \url{http://oeis.org}, 2014,
  Sequence \href{https://oeis.org/A059427}{A059427}.

\bibitem{Stanley2001}
Richard~P. Stanley, \emph{Enumerative {C}ombinatorics, vol. 2}, Cambridge
  University Press, 2001.

\bibitem{st-alt}
Richard~P. Stanley, \emph{A survey of alternating permutations}, Combinatorics
  and graphs, Contemp. Math., vol. 531, Amer. Math. Soc., Providence, RI, 2010,
  pp.~165--196. \MR{2757798 (2012d:05015)}

\end{thebibliography}

\end{document}